\documentclass[reqno]{amsart}
\usepackage{amssymb}
\usepackage[dvips]{epsfig}
\usepackage{graphicx}
\usepackage{color}
\usepackage{amsmath}
\usepackage{amssymb}         
\usepackage{amsfonts}
\usepackage{amsthm}
\usepackage{bbm}

\newcommand{\R}{\mathbb R}

\newcommand{\Z}{\mathbb Z}

\newcommand{\Ima}{{\rm Im}}

\usepackage[colorlinks = true,
linkcolor = blue,
urlcolor  = blue,
citecolor = blue,
anchorcolor = blue]{hyperref}
\usepackage{hyperref}

\newtheorem{theorem}{Theorem}[section]
\newtheorem*{theorem-a}{Theorem A}
\newtheorem*{theorem-b}{Theorem B}

\newtheorem{corollary}[theorem]{Corollary}
\theoremstyle{remark}
\newtheorem{remark}{Remark}[section]
\theoremstyle{definition}

\numberwithin{equation}{section}

\begin{document}

\title[Asymptotic behavior]{On long time behavior of solutions of the Schr\"{o}dinger-Korteweg-de Vries system.}
\author{F. Linares}
\address[F. Linares] {IMPA\\ Estrada Dona Castorina 110, Rio de Janeiro 22460-320, Brazil}
\email{linares@impa.br}
\author{A.J. Mendez}
\address[A. J. Mendez] {Centro de Modelamiento Matem\'atico, Universidad de Chile, Santiago de Chile, Chile}
\email{amendez@dim.uchile.cl }


\begin{abstract}  We are concerned with the decay of long time solutions of the initial value problem associated to the
Schr\"odinger-Korteweg-de Vries system. We use recent techniques in order to show that solutions of this system 
decay to zero in the energy space. The result is independent of the integrability of the equations involved and it does not 
require any size assumptions.
\end{abstract}

\maketitle

\section{Introduction}

We consider the Schr\"odinger-Korteweg-de Vries system,

\begin{equation}\label{s-kdv}
\left\{
\begin{array}{ll}
i\partial_{t}u+\partial_{x}^{2} u=\alpha uv +\beta u|u|^{2}, & \hskip15pt x, t \in  \mathbb{R}, \\
\partial_{t}v+ \partial_{x}^{3}v+v\partial_{x}v=\gamma \partial_{x}\left(|u|^{2}\right),&  \\
u(x,0)=u_{0}(x), v(x,0)=v_{0}(x),
\end{array} 
\right.
\end{equation}
where $u=u(x,t)$ is a complex-valued function, $v=v(x,t)$ is a real-valued function and $\alpha,\beta,\gamma$ are  real constants. 

\medskip

The system in \eqref{s-kdv} appears as a particular case (under appropriate transformations) of the more general system 
\begin{equation*}
\begin{cases}
i\partial_t S+ic_S\,\partial_x S+\partial_x^2 S=\alpha SL+\gamma\vert S\vert^2S,\\ 
\\ \partial_t L+c_L\,\partial_x L+\nu P(D_x)L+\lambda \partial_x L^2=\beta \partial_x \vert S\vert^2,
\end{cases}
\end{equation*}
where $S$ is a complex-valued function representing the short wave, $L$ is a real-valued function representing the long wave and $P(D_x)$ is a differential operator with constant coefficients and $c_S, c_L,\nu,\lambda,\beta, \alpha,\gamma$ are real parameters, see \cite{BeOgPo, LiPa} and references therein. 
This system has received considerable attention because of the broad diversity of physical settings in which it arises. 
For instance, as modelling the internal gravity-wave packet and the capillary-gravity interaction wave see  \cite{DjRe, Gr, KaSuKa}). Furthermore, when $\gamma=0$ the previous system has been derived as a model for the resonant ion-sound/Langmuir wave interaction in plasma physics under the assumption that the ion-sound wave is unidirectional (see \cite{Ma,NiHoMiIk}). Moreover, this system appears in the general theory of water wave interaction in a nonlinear medium. Finally, this system also occurs as a model for the motion of two fluids under capillary-gravity waves in a deep water flow (see \cite{FuOi}) or the motion of two fluids under a shallow water flow (see \cite{FuOi}).

\medskip

The Schr\"odinger-Korteweg-de Vries system \eqref{s-kdv} has been shown not to be a completely integrable system (see \cite{BeBu}). Even though their
solutions satisfy the following conserved quantities, 
\begin{equation}\label{e1}
\begin{split}
&I_{1}[t]:=\int_{\mathbb{R}}|u|^{2}\,\mathrm{d}x= I_{1}[0],\\
&I_{2}[t]:=\int_{\mathbb{R}}\left\{ \alpha\gamma v|u|^{2}-\frac{\alpha}{6}v^{3}+\frac{\beta\gamma}{2}|u|^{4}+\frac{\alpha}{2}|\partial_{x}v|^{2}+ \gamma|\partial_{x}u|^{2}\right\}\,\mathrm{d}x= I_{2}[0],\\
&I_{3}[t]:=\int_{\mathbb{R}}\left\{\alpha v^{2}+2\gamma\Ima\left(u\overline{\partial_{x}u}\right)\right\}\,\mathrm{d}x=I_{3}[0].
\end{split}
\end{equation}

\medskip

The IVP \eqref{s-kdv} has been extensively studied from the view point of local and global well-posedness. This is mainly because of the
wide research developed for the famous Korteweg-de Vries (KdV) equation
$$
v_t+v_{xxx}+vv_x=0
$$
and the cubic Schr\"odinger (NLS) equation 
$$
iu_t+u_{xx}+|u|^2u=0,
$$
for optimal results see \cite{KePoVe2}, \cite{KV},  \cite{Ts}. For further references  see \cite{Bou}, \cite{LiPo}, \cite{SS}, and \cite{Tao}.
This has served of source of inspiration for several authors to study the IVP \eqref{s-kdv}. In general, a coupled system like \eqref{s-kdv} is more 
difficult to handle in the same spaces as in the space the single equation is solved. In the case of the system \eqref{s-kdv} this is due to the 
antisymmetric nature of the characteristics of each linear part. The first result recorded concerning local well-posedness is due to M. Tsutsusmi \cite{T} 
where the author used the smoothing properties for the KdV and NLS equations to proved the local well-posedness for IVP \eqref{s-kdv} for data $(u_0,v_0)\in H^{s+\frac12}(\R)\times H^s(\R)$ for $s\in \Z^{+}$. In \cite{BeOgPo2} Bekiranov, Ogawa and Ponce employed the Fourier restriction method showed that the coupled system \eqref{s-kdv} is locally well-posed in $H^s(\mathbb{R})\times H^{s-\frac{1}{2}}(\mathbb{R})$ with $s\geq 0$. In \cite{CL} Corcho and Linares extended this result for weak initial data $(u_0,v_0)\in H^k(\mathbb{R})\times H^s(\mathbb{R})$ for various values of $k$ and $s$, where the lowest admissible values are $k=0$ and $s=-\tfrac{3}{4}+\delta$ with $0<\delta\leq \tfrac{1}{4}$. The end-point $(k,s)=(0,-\tfrac{3}{4})$ was treated in \cite{GuWa} 
by Z. Guo and Y. Wang. In \cite{Wu} Wu extended the local results in \cite{CL} for $\beta=0$ in \eqref{s-kdv}.

\medskip

Regarding global well-posedness, the conserved quantities were used in \cite{T} to extend the local theory globally for initial data in $H^{s+\frac12}(\R)\times H^s(\R)$ for $s\in \Z^{+}$. Global well-posedness in  the energy space $H^1(\R)\times H^1(\R)$ was established in \cite{CL}. Pecher  in \cite{Pe} proved global results in $H^s(\R)\times H^s(\R)$, $s>3/5$, for $\beta=0$ and $s>2/3$ for $\beta\neq 0$ by using the I-method. Finally,  Wu  (\cite{Wu}) shows global well-possedness in $H^s(\R)\times H^s(\R)$, $s>1/2$, for any $\beta\in \R$.

\medskip
 
We also notice that local/global well-posedness results in weighted Sobolev spaces are also known for this system (\cite{LiPa}).

\medskip

For further reference we present the local and global well-posedness theory obtained in \cite{CL}

\begin{theorem-a}[\cite{CL}] Let $k\ge 0$ and $s>-\frac34$. Then for any $(u_0, v_0)\in H^k(\R)\times H^s(\R)$ provided:
\begin{itemize}
\item[(i)] $k-1\le s\le 2k-\frac12$ for $k\in [0,\frac12]$,
\item[(ii)] $k-1\le s< k+\frac12$ for $k\in (\frac12,\infty)$,
\end{itemize}
there exist a positive time $T=T(\|u_0\|_{H^k}, \|v_0\|_{H^s})$ and a unique solution $(u(t), v(t))$ of the IVP \eqref{s-kdv}, satisfying
\begin{equation}\label{TA-1}
\varphi_T(t)u\in X^{k, \frac12+} \hskip10pt\text{and}\hskip10pt \varphi_T(t)u\in Y^{k, \frac12+},
\end{equation}
\begin{equation}\label{TA-2}
u\in C([0,T]; H^k(\R))  \hskip10pt\text{and}\hskip10pt  v\in C([0,T]; H^s(\R)).
\end{equation}

Moreover, the map $(u_0, v_0) \mapsto (u(t), v(t))$ is locally Lipschitz from $H^k(\R)\times H^s(\R)$ into $C([0,T]; H^k(\R)\times H^s(\R))$.
\end{theorem-a}

The spaces $X^{k, \frac12+}$ and $Y^{k, \frac12+}$ will be defined below. 

\medskip 

Regarding the global theory we have.

\begin{theorem-b}[\cite{CL}]\label{th1} Let $\alpha, \beta, \gamma \in \R$ such that $\alpha\cdot\gamma>0$ and $(u_0,v_0)\in H^1(\R)\times H^1(\R)$.
Then, the unique solution provided by Theorem A can be extended for any time $T>0$. Moreover,
\begin{equation}\label{global-bound}
\sup_{t}\big (\|u(t)\|_{H^1}+\|v(t)\|_{H^1}\big) \le \Psi(\|u_0\|_{H^1},\|v_0\|_{H^1}),
\end{equation}
where $\Psi$ is a function only depending on $\|u_0\|_{H^1}$ and $\|v_0\|_{H^1}$.
\end{theorem-b}

\medskip

Next, we comment on solitary wave solutions or ground states for system \eqref{s-kdv}.  Solitary wave solutions of the
form
\begin{equation}\label{sws}
(u(x,t), v(x,t))=(e^{i\omega t} e^{ic(x-ct)}\phi(x-ct),\psi(x-ct)),
\end{equation}
have been found for the system \eqref{s-kdv}. In \cite{AlAn} the authors
consider the system with $\beta=0$. They show the existence of solitary wave solutions and prove the orbital stability of these
solution for the parameters $\alpha, \gamma<0$. The existence of  ground states for the system with $\beta\neq 0$ was established
in \cite{DFO1} again for the parameters $\alpha, \gamma<0$. The stability of these ground states was proved in \cite{AlBs}, (see also
\cite{CPS}, \cite{DFO2}).

\medskip

\medskip

In particular, for a suitable parameter $c^{*}$, $ \alpha\in (-\frac16, 0)$ and $c=4c^{*}-\frac{1}{12}\alpha(1+6\alpha)$,
the authors in \cite{DFO1} found the following explicit solution for the system \eqref{s-kdv}, 
\begin{equation*}
\phi(x)=\frac{\sqrt{2c^{*}(1+6\alpha)}}{\cosh (\sqrt{c^{*}}\,x)}, \hskip10pt\text{and}\hskip10pt
\psi(x)=\frac{12c^{*}}{\cosh^2(\sqrt{c^{*}}\,x)}.
\end{equation*}

\medskip 

In this work, we are interested in the asymptotic behaviour of solutions to the system \eqref{s-kdv}. As far as we know there are not
results regarding this issue for the global solutions obtained in \cite{CL} for this system. We will show decay properties for solutions of the
system \eqref{s-kdv} as time evolves. 

\medskip 

Our work is inspired by recent results obtained for several dispersive equations and systems of dispersive equations and the techniques
implemented to establish those results. We shall mention the works of Mu\~noz and Ponce for the generalised KdV equation (\cite{MP1}) and the
Benjamin-Ono (BO) equation (\cite{MP2}). Where the authors show the decay of the solutions to these equations. In particular, their results rule out
the possibility to have breathers solutions. The methods used were motivated by the works of Kowalczyk, Martel and Mu\~noz \cite{KMM}-\cite{KMM1}.
In  \cite{LiMePo}, Linares, Mendez and Ponce used the approach developed in \cite{MP2} to study the the decay of solutions to the dispersion 
generalized BO equation.  New techniques were introduced by Mu\~noz, Ponce and Saut \cite{MuPoSa} to investigate the long time behavoir issue
for solutions to the Intermediate Long Wave (IWL) equation. Extensions to higher dimensional model as  the Zakharov-Kusnetsov (ZK) equation were
made by Mendez, Mu\~noz, Poblete and Pozo in \cite{MMPP}.

 Martinez in \cite{MM} considered the 1-dimensional Zakharov system
\begin{equation}\label{ZS}
\begin{cases}
iu_t+u_{xx}=nu,\hskip25pt x,t\in \R,\\
n_{tt}-n_{xx}=(|u|^2)_{xx}.
\end{cases}
\end{equation}
Under some smallness assumptions it is proved that in compact sets the solutions of \eqref{ZS} decay to zero as time tends to infinity. In addition, the author established decay in far field regions along curves.

\section{Main Results}

Our first result regards the decay of global solutions in $H^1(\R)\times H^1(\R)$ in the $L^2-$norm. More precisely.

\begin{theorem}\label{main}
Let $\alpha, \gamma <0$.  Suppose  that  	 $(u_{0},v_{0})\in H^{1}(\mathbb{R})\times H^1(\mathbb{R}).$   Then, the corresponding solution $(u,v) $  to \eqref{s-kdv} with initial condition $(u_{0},v_{0})$ satisfies
	 \begin{equation*}
	 \liminf_{t\uparrow \infty} \int_{\Omega_{p_{1}}(t)} v^{2}(t,x) \mathrm{d}x=	 \liminf_{t\uparrow \infty} \int_{\Omega_{p_{1}}(t)} |u|^{2}(t,x)\,\mathrm{d}x=0,
	 \end{equation*}
	 where 
	 \begin{equation}\label{set}
	 \Omega_{p_{1}}(t)=\{x\in\mathbb{R}\,|\, |x|\lesssim t^{p_{1}} \}\quad\mbox{with}\quad 0<p_{1}< \frac{2}{3}.
	 \end{equation}
	\end{theorem}
	
As as corollary of Theorem \ref{main} we obtain the following  result  in the non-centered case.
\begin{corollary}\label{cor-main}
Let $\alpha, \gamma$  be real numbers such that  $\alpha,\gamma<0.$ Suppose  that  $(u_{0},v_{0})\in H^{1}(\mathbb{R})\times H^{1}(\mathbb{R}).$   
Then, the corresponding solution $(u,v) $  to \eqref{s-kdv} with initial condition $(u_{0},v_{0})$ satisfies 
\begin{equation*}
\liminf_{t\uparrow \infty} \int_{\Gamma_{p_{1}}(t)} (v(t,x))^{2}\mathrm{d}x=	\liminf_{t\uparrow \infty} \int_{\Gamma_{p_{1}}(t)} \left|u(t,x)\right|^{2}\,\mathrm{d}x=0,
\end{equation*}
where 
\begin{equation*}
\Gamma_{p_{1}}(t)=\{x\in\mathbb{R}\,|\, |x-t^{m}|\lesssim t^{p_{1}} \},\quad 0<p_{1}< \frac{2}{3} \quad \mbox{and}\quad 0<m<1-\frac{p_{1}}{2}.
\end{equation*}
\end{corollary}

\medskip

The second main result in this work tell us about the decay of the first derivatives in the $L^2-$norm of solutions in $H^1(\R)\times H^1(\R)$ of the
IVP \eqref{s-kdv}.	
\begin{theorem}\label{main3}
	Let $\alpha, \gamma$  be real numbers such that  $\alpha,\gamma<0.$ Suppose  that  $(u_{0},v_{0})\in H^{1}(\mathbb{R})\times H^{1}(\mathbb{R}).$   Then, the corresponding solution $(u,v) $  to \eqref{s-kdv} with initial condition $(u_{0},v_{0})$ satisfies: 
	\begin{itemize}
		\item[(i)] If $\beta>0,$ then 
	\begin{equation}
	\liminf_{t\uparrow \infty} \int_{\Omega_{p_{1}}(t)} (\partial_{x}v(t,x))^{2}\mathrm{d}x=	\liminf_{t\uparrow \infty} \int_{\Omega_{p_{1}}(t)} \left|\partial_{x}u(t,x)\right|^{2}\,\mathrm{d}x=0,
	\end{equation}
	and 
	\begin{equation}\label{e.1}
		\liminf_{t\uparrow \infty} \int_{\Omega_{p_{1}}(t)} |u(t,x)|^{4}\mathrm{d}x=0,
	\end{equation}
	where 
	\begin{equation}\label{set2}
	\Omega_{p_{1}}(t)=\{x\in\mathbb{R}\,|\, |x|\lesssim t^{p_{1}} \}\quad\mbox{with}\quad 0<p_{1}< \frac{2}{3}. 
	\end{equation}
	\item[(ii)] If $\beta\leq0,$ then 
	\begin{equation*}
	\liminf_{t\uparrow \infty} \int_{\Omega_{p_{1}}(t)} (\partial_{x}v(t,x))^{2}\mathrm{d}x=	\liminf_{t\uparrow \infty} \int_{\Omega_{p_{1}}(t)} \left|\partial_{x}u(t,x)\right|^{2}\,\mathrm{d}x=0,
	\end{equation*}
	where $\Omega_{p_{1}}(t)$ is as in \eqref{set2}.
\end{itemize}
\end{theorem}
	Some remarks are in order.
	\begin{remark} Combining the results in Theorem \ref{main} and Theorem \ref{main3} we guarantee the decay
	in the energy space of solutions of the IVP \eqref{s-kdv} for the parameters $\alpha, \gamma<0$.
	\end{remark}
	
	\begin{remark} The approach we follows is closer to the one introduced in \cite{MMPP}.  This method shows to be independent
	of the integrability of the equation and does not need size restriction.
	\end{remark}
	
	\begin{remark} The results above rule out the possibility to have small breather solutions.
	\end{remark}
	
	\begin{remark}
	The idea of the proof of Corollary \ref{cor-main} follows  the argument used in \cite{MMPP}. Thus, it  is enough to define the  functional
	\begin{equation*}
	\mathcal{J}(t):=\frac{1}{\eta(t)}\int_{\mathbb{R}} u(t,x)\varphi_{a}\left(\frac{x-t^{m}}{\lambda_{1}(t)}\right)\phi_{b}\left(\frac{x-t^{m}}{\lambda_{2}(t)}\right)\,\mathrm{d}x
	\end{equation*}
	and then proceed with the analysis described in the proof of Theorem \ref{main}.
\end{remark}
	
%	\begin{remark} The analysis becomes simpler if the parameter $\beta=0$ in \eqref{s-kdv}.
%	\end{remark}

	\begin{remark} The case when the parameters $\alpha$ and $\gamma$ are both positive is open.
More precisely,   following the notation introduced in the proof of Theorem \ref{main} (see  Section \ref{proofmain1}) we obtain by means of energy estimates the  bound
\begin{equation}\label{e6.11}
\begin{split}
&\int\limits_{\{t\gg1\}}\frac{1}{\eta(t)\lambda_{1}(t)}\int_{\mathbb{R}}\left(\frac{v^{2}}{4}-\gamma|u|^{2}\right)\varphi_{a}'\Big(\frac{x}{\lambda_{1}(t)}\Big)\phi_{b}\left(\frac{x}{\lambda_{2}(t)}\right)\,\mathrm{d}x
<\infty,
\end{split}
\end{equation}
where $a,b>0.$ However,  since there is  not a defined   sign in the inequality \eqref{e6.11}  we cannot decide  if there is some decay  of the $L^{2}-$mass norms of $u,v.$

We observe from the comments regarding solitary waves that in this case the existence of such
	solutions is unknown.
	\end{remark}
	
\medskip
	
The remainder of this paper is structured as follows, Section 2 will be devoted to prove Theorem \ref{main} and Section 3 will present the
proof of Theorem \ref{main3}. Before leaving this section we will set the notation we employ in this manuscript.	

\vspace{2mm}	
\noindent{\bf Notation.} For $k, s\in \R$ and $b \in (0,1)$ we let $X^{k,s}$ and $Y^{s,b}$ be the completion of $\mathcal{S}(\R^2)$ 
with respect to the norms
$$
\|f\|_{X^{k,b}}=\Big(\int\int \langle \xi\rangle^{2k}\langle \tau+\xi^2\rangle^{2b} |\widehat{f}(\xi,\tau)|^2\,d\tau d\xi\Big)^{1/2}
$$
and
$$
\|f\|_{Y^{s,b}}=\Big(\int\int \langle \xi\rangle^{2s}\langle \tau-\xi^3\rangle^{2b} |\widehat{f}(\xi,\tau)|^2\,d\tau d\xi\Big)^{1/2} 
$$
where $ \langle \cdot \rangle= (1+|\xi|^2)^{1/2}$ and $\widehat{f}(\xi,\tau)$ denotes the Fourier transform in the $x,t$ variables.

\vspace{1cm}

\section{Proof of Theorem \ref{main}}\label{proofmain1}

 \begin{proof}[Proof of Theorem \ref{main}]
 	
	 We follow the argument of proof presented in \cite{MMPP}. Then, we will consider the following  weighted  functions:
 \begin{equation}\label{e3}
 \varphi(x)=\frac{2}{\pi}\arctan(e^{x}),\hskip10pt x\in\mathbb{R},
 \end{equation}
 and 
  \begin{equation}\label{e4}
 \varphi'(x)=\frac{1}{\pi \cosh\left(x\right)}\sim e^{-|x|},\qquad |\varphi'''(x)|\lesssim \varphi'(x),\quad x\in\mathbb{R}.
 \end{equation}
 For $a>0,$ we  define   $\varphi_{a}(x):=a\varphi\left(\frac{x}{a}\right)$ and $\phi_{a}(x):= \varphi'\left(\frac{x}{a}\right).$

 For  $p_{1},r_{1},r_{2},q_{1},q_{2}>0$  we set
 \begin{equation}\label{e5}
 \lambda_{1}(t)=\frac{t^{p_{1}}}{\ln^{q_{1}}t} ,\quad  \lambda_{2}(t)=(\lambda_{1}(t))^{p_{2}}\quad\mbox{and}\quad \eta(t)=t^{r_{1}}\ln^{r_{2}}t, 
 \end{equation}
 whence  
 \begin{equation*}
 p_{1}+r_{1}=1,\quad p_{2}>1\quad \mbox{and}\quad  r_{2}=1+q_{1}. 
 \end{equation*}
 Then
 \begin{equation*}
 \frac{\lambda_{1}'(t)}{\lambda_{1}(t)}\sim \frac{1}{t},\quad  \frac{\lambda_{2}'(t)}{\lambda_{2}(t)}\sim \frac{1}{t}\quad \mbox{and}\quad  \frac{\eta'(t)}{\eta(t)}\sim\frac{1}{t}\quad \mbox{for}\quad t\gg1 .
 \end{equation*}
 For $v=v(t,x)$ solution to \eqref{s-kdv} we define  the functional
 \begin{equation}\label{functional-1}
 \mathcal{J}(t):=\frac{1}{\eta(t)}\,\int_{\mathbb{R}}v(t,x)\varphi_{a}\left(\frac{x}{\lambda_{1}(t)}\right)\phi_{b}\left(\frac{x}{\lambda_{2}(t)}\right)\,\mathrm{d}x,
 \end{equation}
 where $a,b>0.$
 
 We claim that $\mathcal{J}$ is well defined.  In this sense, we have  
 by H\"{o}lder's inequality 
 \begin{equation*}
 \begin{split}
 |\mathcal{J}(t)|\leq \frac{(\lambda_{2}(t))^{1/2}\|v(t)\|_{2}}{\eta(t)}\left\|\varphi_{a}\left(\frac{\cdot}{\lambda_{1}(t)}\right)\right\|_{\infty}.
 \end{split}
 \end{equation*}
Thus, by \eqref{e1}
\begin{equation*}
\begin{split}
|\mathcal{J}(t)|&\lesssim \frac{(\lambda_{2}(t))^{1/2}\|v(t)\|_{2}}{\eta(t)}\\
&\lesssim \frac{(\lambda_{2}(t))^{1/2}}{\eta(t)}\left(\frac{I_{3}[0]}{|\alpha|}+\frac{2\gamma}{\alpha}\|u_{0}\|_{2}\|\partial_{x}u(t)\|_{2}\right).
\end{split}
\end{equation*}
 Next, from Theorem \ref{th1} we get 
 \begin{equation*}
 \sup_{t\gg 1}|\mathcal{J}(t)|<\infty,
 \end{equation*}
whenever 
\begin{equation}\label{l1}
0<p_{1}\leq \frac{2}{p_{2}+2}.
\end{equation}
 In the following we will suppress the dependence on the variables in order to simplify the notation unless  it would  be necessary.
 
 Next,
 \begin{equation}\label{energy}
 \begin{split}
  \frac{\mathrm{d}}{\mathrm{d}t}\mathcal{J}(t)
 &=\underbrace{\frac{1}{\eta(t)}\int_{\mathbb{R}}\partial_{t}v\varphi_{a}\left(\frac{x}{\lambda_{1}(t)}\right)\phi_{b}\left(\frac{x}{\lambda_{2}(t)}\right)\,\mathrm{d}x}_{A_{1}(t)}\\
 &\underbrace{-\frac{\lambda_{1}'(t)}{\eta(t)\lambda_{1}(t)}\int_{\mathbb{R}}v\varphi_{a}'\left(\frac{x}{\lambda_{1}(t)}\right)\left(\frac{x}{\lambda_{1}(t)}\right)\phi_{b}\left(\frac{x}{\lambda_{2}(t)}\right)\,\mathrm{d}x}_{A_{2}(t)}\\
 &\underbrace{-\frac{\lambda_{2}'(t)}{\eta(t)\lambda_{2}(t)}\int_{\mathbb{R}}v\varphi_{a}\left(\frac{x}{\lambda_{1}(t)}\right)\left(\frac{x}{\lambda_{1}(t)}\right)\phi_{b}'\left(\frac{x}{\lambda_{2}(t)}\right)\,\mathrm{d}x}_{A_{3}(t)}\\
 & \underbrace{-\frac{\eta'(t)}{\eta^{2}(t)}\int_{\mathbb{R}}v\varphi_{a}\left(\frac{x}{\lambda_{1}(t)}\right)\phi_{b}\left(\frac{x}{\lambda_{2}(t)}\right)\,\mathrm{d}x}_{A_{4}(t)}.
 \end{split}
 \end{equation}
 
For $A_{1}$  we  combine \eqref{s-kdv} and integration by parts whence we obtain 
\begin{equation}
\begin{split}
A_{1}(t)
&=-\frac{\gamma}{\lambda_{1}(t)\eta(t)}\int_{\mathbb{R}}|u|^{2}\varphi_{a}'\left(\frac{x}{\lambda_{1}(t)}\right)\phi_{b}\left(\frac{x}{\lambda_{2}(t)}\right)\mathrm{d}x\\
&\quad  -\frac{\gamma}{\eta(t)\lambda_{2}(t)}\int_{\mathbb{R}}|u|^{2}\varphi_{a}\left(\frac{x}{\lambda_{1}(t)}\right)\phi_{b}'\left(\frac{x}{\lambda_{2}(t)}\right)\mathrm{d}x\\
&\quad  -\frac{1}{2\eta(t)\lambda_{1}(t)}\int_{\mathbb{R}}v^{2}\varphi_{a}'\left(\frac{x}{\lambda_{1}(t)}\right)\phi_{b}\left(\frac{x}{\lambda_{2}(t)}\right)\mathrm{d}x\\
&\quad -\frac{1}{2\eta(t)\lambda_{2}(t)}\int_{\mathbb{R}}v^{2}\varphi_{a}\left(\frac{x}{\lambda_{1}(t)}\right)\phi_{b}'\left(\frac{x}{\lambda_{2}(t)}\right)\mathrm{d}x\\
&\quad  +\frac{1}{\eta(t)\lambda_{1}^{3}(t)}\int_{\mathbb{R}}v \varphi_{a}'''\left(\frac{x}{\lambda_{1}(t)}\right)\phi_{b}\left(\frac{x}{\lambda_{2}(t)}\right)\mathrm{d}x\\
&\quad  +\frac{3}{\eta(t)\lambda_{1}^{2}(t)\lambda_{2}(t)}\int_{\mathbb{R}}v \varphi_{a}''\left(\frac{x}{\lambda_{1}(t)}\right)\phi_{b}'\left(\frac{x}{\lambda_{2}(t)}\right)\mathrm{d}x\\
&\quad +\frac{3}{\eta(t)\lambda_{1}(t)\lambda_{2}^{2}(t)}\int_{\mathbb{R}}v \varphi_{a}'\left(\frac{x}{\lambda_{1}(t)}\right)\phi_{b}''\left(\frac{x}{\lambda_{2}(t)}\right)\mathrm{d}x\\
&\quad  +\frac{1}{\eta(t)\lambda_{2}^{3}(t)}\int_{\mathbb{R}}v \varphi_{a}\left(\frac{x}{\lambda_{1}(t)}\right)\phi_{b}'''\left(\frac{x}{\lambda_{2}(t)}\right)\mathrm{d}x\\
&=A_{1,1}(t)+A_{1,2}(t)+A_{1,3}(t)+A_{1,4}(t)+A_{1,5}(t)+A_{1,6}(t)\\
&\quad +A_{1,7}(t)+A_{1,8}(t).
\end{split}
\end{equation}
Up to a constant, the term $A_{1,1}$ is the quantity to be estimated  after integrating in time. Instead, for $A_{1,2}$ we have the following bound 
\begin{equation}\label{p1}
|A_{1,2}(t)|\lesssim_{\gamma, \|u_{0}\|_{2}} \frac{1}{\lambda_{2}(t)\eta(t)}\in L^{1}\left(\{t\gg 1\}\right),
\end{equation}
 since $p_{2}>1.$
 
 The term $A_{1,3}$ corresponds to the contribuiton coming from the KdV dynamics that we  want to   estimate after integrating in time. Nevertheless,   for  $A_{1,4}$ the situation is quite differet  and we  shall proceed in a different manner.
 
 By using   H\"{o}lder's inequality we get 
 \begin{equation}\label{p2}
 \begin{split}
 |A_{1,4}(t)|&\lesssim_{\gamma,\alpha,\|u_{0}\|_{2},\|u\|_{L^{\infty}_{t}H^{1}}} \frac{1}{\eta(t)\lambda_{2}(t)}\in L^{1}\left(\{t\gg 1\}\right),
 \end{split}
 \end{equation} 
 since $p_{2}>1.$

Next, we have 
\begin{equation}\label{p3}
\begin{split}
|A_{1,5}(t)|&\lesssim_{\alpha,\gamma,\|u_{0}\|_{2},\|u\|_{L^{\infty}_{t}H^{1}}}\frac{1}{\eta(t)\lambda_{1}^{5/2}(t)}\in L^{1}\left(\{t\gg 1\}\right),
\end{split}
\end{equation} 
 since $p_{1}>0.$
 
 For $A_{1,6}$ the  we obtain the following bound
 \begin{equation}\label{p4}
 |A_{1,6}(t)|\lesssim_{\alpha,\gamma,\|u_{0}\|_{2},\|u\|_{L^{\infty}_{t}H^{1}}} \frac{1}{\eta(t)\lambda_{1}^{3/2}(t)\lambda_{2}(t)}\in L^{1}\left(\{t\gg 1\}\right)
 \end{equation}
 whenever $p_{1},p_{2}>0.$
 
  For $A_{1,7}$ we have 
 \begin{equation}\label{p5}
| A_{1,7}(t)|\lesssim_{\alpha,\gamma,\|u_{0}\|_{2},\|u\|_{L^{\infty}_{t}H^{1}}}\frac{1}{\eta(t)\lambda_{1}^{2p_{2}+\frac{1}{2}}(t)}\in L^{1}\left(\{t\gg 1\}\right),
 \end{equation}
 since $p_{2}>1.$
 
 The last term  coming from $A_{1}$  corresponds to $A_{1,8}$ whence we obtain after applying H\"{o}lder's inequality 
 \begin{equation}\label{p6}
 |A_{1,8}(t)|\lesssim_{\alpha,\gamma,\|u_{0}\|_{2},\|u\|_{L^{\infty}_{t}H^{1}}} \frac{1}{\eta(t)\lambda_{1}^{\frac{5p_{2}}{2}}(t)}\in L^{1}\left(\{t\gg 1\}\right),
 \end{equation}
  since $p_{2}>1.$
  
  Next, we estimate $A_{2}.$ By  Young's inequality we have the following:
  For $\epsilon>0,$
  \begin{equation*}
  \begin{split}
| A_{2}(t)|&=\left| -\frac{\lambda_{1}'(t)}{\eta(t)\lambda_{1}(t)}\int_{\mathbb{R}}v\varphi_{a}'\left(\frac{x}{\lambda_{1}(t)}\right)\left(\frac{x}{\lambda_{1}(t)}\right)\phi_{b}\left(\frac{x}{\lambda_{2}(t)}\right)\,\mathrm{d}x\right|\\
&\leq\frac{1}{4\epsilon}\left| \frac{\lambda_{1}'(t)}{\eta(t)\lambda_{1}(t)}\right|\int_{\mathbb{R}}v^{2}\varphi_{a}'\left(\frac{x}{\lambda_{1}(t)}\right)\phi_{b}\left(\frac{x}{\lambda_{2}(t)}\right)\,\mathrm{d}x\\
&\quad +\epsilon \left|\frac{\lambda_{1}'(t)}{\eta(t)\lambda_{1}(t)}\right|\int_{\mathbb{R}}\varphi_{a}'\left(\frac{x}{\lambda_{1}(t)}\right)\left|\frac{x}{\lambda_{1}(t)}\right|^{2}\phi_{b}\left(\frac{x}{\lambda_{2}(t)}\right)\,\mathrm{d}x,
  \end{split}
  \end{equation*}
whence   after taking  $\epsilon=|\lambda_{1}'(t)|>0$ we obtain  
 \begin{equation*}
\begin{split}
| A_{2}(t)|
&\leq\ \frac{1}{4\eta(t)\lambda_{1}(t)}\int_{\mathbb{R}}v^{2}\varphi_{a}'\left(\frac{x}{\lambda_{1}(t)}\right)\phi_{b}\left(\frac{x}{\lambda_{2}(t)}\right)\,\mathrm{d}x\\
&\quad + \frac{(\lambda_{1}'(t))^{2}}{\eta(t)}\|\phi_{b}\|_{\infty}\|(\cdot )^2\varphi_{a}'(\cdot)\|_{1}.
\end{split}
\end{equation*}
  Note that the first term in the r.h.s is up to constant   the quantity to be estimated after integrating in time. Instead, the second term  in the r.h.s  satisfies 
  \begin{equation*}
  \frac{(\lambda_{1}'(t))^{2}}{\eta(t)}\in L^{1}\left(\{t\gg1\}\right),
  \end{equation*}
  whenever $0<p_{1}<\frac{2}{3}.$
  
  The situation is more delicate when we try to estimate $A_{3},$ so that we will consider the following auxiliar function $\theta(t)=t^{1-p_{1}},$  then 
  by Young's inequality we obtain 
  \begin{equation*}
  \begin{split}
  |A_{3}(t)|&=\left|-\frac{\lambda_{2}'(t)}{\eta(t)\lambda_{2}(t)}\int_{\mathbb{R}}v\varphi_{a}\left(\frac{x}{\lambda_{1}(t)}\right)\left(\frac{x}{\lambda_{2}(t)}\right)\phi_{b}'\left(\frac{x}{\lambda_{2}(t)}\right)\,\mathrm{d}x\right|\\
&\lesssim   \left|\frac{\lambda_{2}'(t)\theta(t)}{\eta(t)\lambda_{2}(t)}\right
|\int_{\mathbb{R}}v^{2}\left|\varphi_{a}\left(\frac{x}{\lambda_{1}(t)}\right)\right|^{2}\mathrm{d}x\\
&\quad + \left|\frac{\lambda_{2}'(t)}{\eta(t)\theta(t)\lambda_{2}(t)}\right|\int_{\mathbb{R}}\left|\frac{x}{\lambda_{2}(t)}\right|^{2}\left|\phi_{b}'\left(\frac{x}{\lambda_{2}(t)}\right)\right|^{2}\mathrm{d}x\\
&\lesssim_{\alpha,\gamma,\|u_{0}\|_{2},\|u\|_{L^{\infty}_{t}H^{1}}} \left|\frac{\theta(t)}{t\eta(t)}\right|+\left|\frac{\lambda_{2}(t)}{t\eta(t)\theta(t)}\right|\in L^{1}\left(\{t\gg1\}\right),
  \end{split}
  \end{equation*}
  since $0<p_{1}\leq \frac{2}{p_{2}+2}$ and $r_{2}>1.$
  
 To handle $A_{4}$ we combine H\"{o}lder's inequality, Theorem \ref{th1}, \eqref{e1} and  \eqref{l1} to obtain  
 \begin{equation*}
 \begin{split}
  A_{4}(t)&\leq  \left|\frac{\eta'(t)(\lambda_{2}(t))^{1/2}}{\eta^{2}(t)}\right|\|v(t)\|_{2}\|\phi_{a}\|_{2}\|\varphi\|_{\infty}\\
  &\lesssim_{\gamma,\alpha ,\|u_{0}\|_{2}} \left|\frac{\eta'(t)(\lambda_{2}(t))^{1/2}}{\eta^{2}(t)}\right|
  \in L^{1}\left(\{t\gg1\}\right),
 \end{split}
 \end{equation*}  
 since $0<p_{1}\leq \frac{2}{p_{2}+2}$ and $r_{2}>1.$

Finally, we get after integrating in time that
 \begin{equation}\label{e6.1}
 \begin{split}
 &\int\limits_{\{t\gg1\}}\frac{1}{\eta(t)\lambda_{1}(t)}\int_{\mathbb{R}}v^{2}\varphi_{a}'\Big(\frac{x}{\lambda_{1}(t)}\Big)\phi_{b}\left(\frac{x}{\lambda_{2}(t)}\right)\,\mathrm{d}x\,\mathrm{d}t\\
 &\quad +\int\limits_{\{t\gg1\}}\frac{1}{\eta(t)\lambda_{1}(t)}\int\limits_{\mathbb{R}}|u|^{2}\varphi_{a}'\Big(\frac{x}{\lambda_{1}(t)}\Big)\phi_{b}\left(\frac{x}{\lambda_{2}(t)}\right)\,\mathrm{d}x\,\mathrm{d}t
<\infty,
 \end{split}
 \end{equation}
 for any $a,b>0.$
 
Since $\frac{1}{\eta(t)\lambda_{1}(t)}=\frac{1}{t\ln t}\notin L^{1}\left(\left\{t\gg1\right\}\right)$  we can guarantee that there exist a sequence of positive time $(t_{n})_{n\geq 1}$  with $t_{n}\uparrow \infty $ as $n$ goes to inifinity such that 
\begin{equation*}
\lim_{n\rightarrow \infty}\int_{\Omega_{p_{1}}(t_{n})}v^{2}(t_{n},x)\,\mathrm{d}x=\lim_{n\rightarrow \infty}\int\limits_{\Omega_{p_{1}}(t_{n})}|u|^{2}(t_{n},x)\,\mathrm{d}x=0,
\end{equation*}
where $\Omega_{p_{1}}(t)$ is as in \eqref{set}.

\end{proof}
\section{Proof of Theorem \ref{main3} }
\begin{proof}
We start by defining the functional 
\begin{equation}\label{functional-2}
\mathcal{I}(t)=\frac{\theta}{2\eta(t)}\int_{\mathbb{R}}v^{2}(t,x)\varphi_{l}\Big(\frac{x}{\lambda_{1}(t)}\Big)\,\mathrm{d}x+\frac{\mu}{\eta(t)}\Ima\int_{\mathbb{R}}u(t,x)\overline{\partial_{x}u(x,t)}\varphi_{l}\Big(\frac{x}{\lambda_{1}(t)}\Big)\mathrm{d}x,
\end{equation}
where $\theta,\mu,l$ are real constant  to be chosen .

We require  $\mathcal{I}$ to  be  well defined.  Indeed, the Cauchy-Schwarz inequality yield
\begin{equation}\label{e7}
\begin{split}
\sup_{t\gg 1}\left|\mathcal{I}(t)\right|&\leq \frac{|\theta|}{2}\|v\|_{2}^{2}\left\|\varphi_{l}\left(\frac{\cdot}{\lambda_{1}(t)}\right)\right\|_{\infty}+|\mu|\|u_{0}\|_{2}\|\partial_{x}u\|_{L^{\infty}_{t}L^{2}_{x}}\left\|\varphi_{l}\left(\frac{\cdot}{\lambda_{1}(t)}\right)\right\|_{\infty}.
\end{split}
\end{equation} 
On the other hand,  from \eqref{e1} we get 
\begin{equation}\label{e8}
\begin{split}
\|v(t)\|_{2}^{2}&\leq \frac{1}{|\alpha|}|I_{3}[0]|+2\frac{|\gamma|}{|\alpha|}\int_{\mathbb{R}}\left| \Ima\left(u(t,x)\overline{\partial_{x}u(t,x)}\right)\right|\\
&\leq  \frac{1}{|\alpha|}|I_{3}[0]|+2\frac{|\gamma|}{|\alpha|}\|u_{0}\|_{2}\|\partial_{x}u(t)\|_{2},\quad \gamma,\alpha\neq 0.
\end{split}
\end{equation}
Therefore, combining \eqref{e7} and \eqref{e8} it follows that
\begin{equation*}
\sup_{t\gg 1} \left|\mathcal{I}(t)\right|<\infty.
\end{equation*}

Next,  combining integration by parts and \eqref{s-kdv} we obtain 
\begin{equation}\label{e2}
\begin{split}
&\frac{\mathrm{d}}{\mathrm{d}t}\mathcal{I}(t)=\\
&\underbrace{\frac{-\theta\gamma}{\eta(t)}\int_{\mathbb{R}}\partial_{x}v|u|^{2}\varphi_{l}\left(\frac{x}{\lambda_{1}(t)}\right)\,\mathrm{d}x}_{B_{1}(t)}\underbrace{-\frac{\theta\gamma}{\eta(t)\lambda_{1}(t)}\int_{\mathbb{R}}v|u|^{2}\varphi_{l}'\left(\frac{x}{\lambda_{1}(t)}\right)\,\mathrm{d}x}_{B_{2}(t)}\\
&\quad \underbrace{+\frac{\theta}{3\eta(t)\lambda_{1}(t)}\int_{\mathbb{R}}v^{3}\varphi_{l}'\left(\frac{x}{\lambda_{1}(t)}\right)\,\mathrm{d}x}_{B_{3}(t)}\underbrace{-\frac{3\theta}{2\eta(t)\lambda_{1}(t)}\int_{\mathbb{R}}\left(\partial_{x}v\right)^{2}\varphi_{l}'\left(\frac{x}{\lambda_{1}(t)}\right)\,\mathrm{d}x}_{B_{4}(t)}\\
&\quad \underbrace{+\frac{\theta}{2\eta(t)\lambda_{1}^{3}(t)}\int_{\mathbb{R}}v^{2}\varphi_{l}'''\left(\frac{x}{\lambda_{1}(t)}\right)\,\mathrm{d}x}_{B_{5}(t)}\underbrace{-\frac{\theta\lambda_{1}'(t)}{2\eta(t)\lambda_{1}(t)}\int_{\mathbb{R}}v^{2}\varphi_{l}'\left(\frac{x}{\lambda_{1}(t)}\right)\left(\frac{x}{\lambda_{1}(t)}\right)\,\mathrm{d}x}_{B_{6}(t)}\\
&\quad \underbrace{-\frac{\theta\eta'(t)}{2\eta^{2}(t)}\int_{\mathbb{R}}v^{2}\varphi_{l}\left(\frac{x}{\lambda_{1}(t)}\right)\,\mathrm{d}x}_{B_{7}(t)} \underbrace{-\frac{\mu\eta'(t)}{\eta^{2}(t)}\Ima\int_{\mathbb{R}}u\overline{\partial_{x}u}\varphi_{l}\left(\frac{x}{\lambda_{1}(t)}\right)\,\mathrm{d}x}_{B_{8}(t)}\\
&\quad \underbrace{-\frac{2\mu}{\eta(t)\lambda_{1}(t)}\int_{\mathbb{R}}\left|\partial_{x}u\right|^{2}\varphi_{l}'\left(\frac{x}{\lambda_{1}(t)}\right)\,\mathrm{d}x}_{B_{9}(t)}\underbrace{+\frac{\alpha\mu}{\eta(t)}\int_{\mathbb{R}}\partial_{x}v|u|^{2}\varphi_{l}\left(\frac{x}{\lambda_{1}(t)}\right)\,\mathrm{d}x}_{B_{10}(t)}\\
&\quad \underbrace{-\frac{\mu\beta}{2\eta(t)\lambda_{1}(t)}\int_{\mathbb{R}}|u|^{4}\varphi_{l}'\left(\frac{x}{\lambda_{1}(t)}\right)\,\mathrm{d}x}_{B_{11}(t)} \underbrace{+\frac{\mu}{2\eta(t)\lambda_{1}^{3}(t)}\int_{\mathbb{R}}|u|^{2}\varphi_{l}'''\left(\frac{x}{\lambda_{1}(t)}\right)\,\mathrm{d}x}_{B_{12}(t)}\\
&\quad+ \underbrace{\frac{\mu\lambda_{1}'(t)}{\eta(t)\lambda_{1}(t)}\Ima\int_{\mathbb{R}}u\overline{\partial_{x}u}\varphi_{l}'\left(\frac{x}{\lambda_{1}(t)}\right)\left(\frac{x}{\lambda_{1}(t)}\right)\,\mathrm{d}x}_{B_{13}(t)}.
\end{split}
\end{equation}
First, we  chose  $\theta>0$ and   $\mu=\frac{\gamma\theta}{\alpha}$ note that  this last  choice  is well defined according to our hypothesis, then
   $$B_{1}+B_{10}=0.$$

Next,    we focus our attention in to provide  the required upper bounds  for the remainder terms.

In this sense,  we have for  $B_{2}$
\begin{equation}\label{e1.1}
\begin{split}
|B_{2}(t)|&=\left|-\frac{\theta\gamma}{\eta(t)\lambda_{1}(t)}\int_{\mathbb{R}}v|u|^{2}\varphi_{l}'\left(\frac{x}{\lambda_{1}(t)}\right)\,\mathrm{d}x\right|\\
&\lesssim \frac{\theta|\gamma|}{\eta(t)\lambda_{1}(t)}\int_{\mathbb{R}} v^{2}\varphi_{l}'\left(\frac{x}{\lambda_{1}(t)}\right)\,\mathrm{d}x+\frac{\theta|\gamma|}{\eta(t)\lambda_{1}(t)}\int_{\mathbb{R}} |u|^{4}\varphi_{l}'\left(\frac{x}{\lambda_{1}(t)}\right)\,\mathrm{d}x\\
&=B_{2,1}(t)+B_{2,2}(t).
\end{split}
\end{equation}
Notice that if we choose $l>0,$ satisfying 
\begin{equation*}
\frac{1}{a}+\frac{1}{b}\leq \frac{1}{l},
\end{equation*}
 we obtain   by \eqref{e6.1} that  $B_{2,1}\in L^{1}\left(\{t\gg 1\}\right),$ while $B_{2,2}$ remains to be estimated. Later on, we will provide the   required arguments to bound this term.

Note that, combining the properties of $\varphi,$ i.e.
\begin{equation*}
|\varphi'''(x)|\lesssim \varphi'(x)\quad\mbox{for all}\quad x\in\mathbb{R},
\end{equation*}
we get 
\begin{equation*}
\begin{split}
|B_{5}(t)|&\lesssim  \frac{\theta}{\eta(t)\lambda_{1}^{3}(t)}\int_{\mathbb{R}}v^{2}\left|\varphi_{l}'''\left(\frac{x}{\lambda_{1}(t)}\right)\right|\,\mathrm{d}x\\
&\lesssim \frac{1}{\eta(t)\lambda_{1}^{3}(t)}\int_{\mathbb{R}}v^{2}\varphi_{l}'\left(\frac{x}{\lambda_{1}(t)}\right)\,\mathrm{d}x\in L^{1}\left(\left\{t\gg 1\right\}\right),
\end{split}
\end{equation*} 
since $p_{1}>0.$

Next, 
\begin{equation*}
|B_{6}(t)|\lesssim_{\alpha,\gamma} \frac{ |I_{3}[0]|}{t^{r_{1}+1}}+\frac{\|u_{0}\|_{2}\|\partial_{x}u\|_{L^{\infty}_{t}L^{2}_{x}}}{ t^{r_{1}+1}}\in L^{1}\left(\{t\gg 1\}\right), 
\end{equation*}
since $r_{1}>0.$
 
 Also, 
 \begin{equation*}
 |B_{7}(t)|\lesssim_{\alpha,\gamma} \frac{ |I_{3}[0]|}{t^{r_{1}+1}}+\frac{\|u_{0}\|_{2}\|\partial_{x}u\|_{L^{\infty}_{t}L^{2}_{x}}}{ t^{r_{1}+1}}\in L^{1}\left(\{t\gg 1\}\right),
 \end{equation*}
and 
\begin{equation*}
|B_{8}(t)|\lesssim \frac{\|u_{0}\|_{2}\|u\|_{L^{\infty}_{t}H^{1}_{x}}}{t^{r_{1}+1}} \in L^{1}\left(\{t\gg 1\}\right)).
 \end{equation*}
Concerninig the term $B_{12}$ we  have that 
\begin{equation*}
|B_{12}(t)|\lesssim \frac{1}{\eta(t)\lambda_{1}^{3}(t)}\int_{\mathbb{R}}|u|^{2}\varphi_{l}'\left(\frac{x}{\lambda_{1}(t)}\right)\,\mathrm{d}x\in L^{1}\left(\{t\gg1\}\right),
\end{equation*}
since $p_{1}>0.$

 We end the first set of estimates bounding $B_{13}$. Indeed,
 \begin{equation*}
 \begin{split}
 |B_{13}(t)|&=\left|\frac{\mu\lambda_{1}'(t)}{\eta(t)\lambda_{1}(t)}\Ima\int_{\mathbb{R}}u\overline{\partial_{x}u}\varphi_{l}'\left(\frac{x}{\lambda_{1}(t)}\right)\left(\frac{x}{\lambda_{1}(t)}\right)\,\mathrm{d}x\right|\\
 &\leq \mu\Big|\frac{\lambda_{1}'(t)}{\eta(t)\lambda_{1}(t)}\Big|\int_{\mathbb{R}}|u||\partial_{x}u|\varphi_{l}'\Big(\frac{x}{\lambda_{1}(t)}\Big)\left|\frac{x}{\lambda_{1}(t)}\right|\,\mathrm{d}x\\
 &\leq \mu\epsilon \Big|\frac{\lambda_{1}'(t)}{\eta(t)\lambda_{1}(t)}\Big|\int_{\mathbb{R}}|u|^{2}\varphi_{l}'\Big(\frac{x}{\lambda_{1}(t)}\Big)
 \left|\frac{x}{\lambda_{1}(t)}\right|^{2}\,\mathrm{d}x\\
 &\hskip10pt +\frac{\mu}{4\epsilon}\Big|\frac{\lambda_{1}'(t)}{\eta(t)\lambda_{1}(t)}\Big|\int_{\mathbb{R}}\big|\partial_{x}u\big|^{2}\varphi_{l}'\Big(\frac{x}{\lambda_{1}(t)}\Big)\,\mathrm{d}x
 \end{split}
 \end{equation*}
 	whence we obtain after taking $\epsilon=\frac{|\lambda_{1}'(t)|}{2}>0,$ the bound 
 	\begin{equation*}
 	\begin{split}
 	|B_{13}(t)|
 	&\leq \frac{\mu}{2\eta(t)\lambda_{1}(t)}\int_{\mathbb{R}} |\partial_{x}u|^{2}\varphi_{l}'\left(\frac{x}{\lambda_{1}(t)}\right)\mathrm{d}x\\
 	&\quad +\frac{\mu}{2}\frac{(\lambda_{1}'(t))^{2}}{\eta(t)\lambda_{1}(t)}\int_{\mathbb{R}} |u|^{2}\varphi_{l}'\left(\frac{x}{\lambda_{1}(t)}\right)\left|\frac{x}{\lambda_{1}(t)}\right|^{2}\mathrm{d}x.
 	\end{split}
 	\end{equation*}
 	Thus, 
 	\begin{equation*}
 	\begin{split}
 	\frac{(\lambda'(t))^{2}}{\eta(t)\lambda_{1}(t)}\int_{\mathbb{R}} |u|^{2}\varphi_{l}'\Big(\frac{x}{\lambda_{1}(t)}\Big)\left|\frac{x}{\lambda_{1}(t)}\right|^{2}\,\mathrm{d}x&\lesssim_{\|u_{0}\|_{2}}\frac{(\lambda_{1}'(t))^{2}}{\eta(t)\lambda_{1}(t)}\in L^{1}\left(\{t\gg1\}\right),
 	\end{split}
 	\end{equation*}
 	since $0<p_{1}<1.$	The  remainder term will be estimated after integrating in time.

Finally, we  bound  the terms that require  a  different  approach to the one used for the previous terms. 

 Note  that, for $\epsilon>0$  the following inequality holds
\begin{equation}\label{e1.2}
\begin{split}
|B_{3}(t)|&=\left|\frac{\theta}{3\eta(t)\lambda_{1}(t)}\int_{\mathbb{R}}v^{3}\varphi'_{l}\left(\frac{x}{\lambda_{1}(t)}\right)\,\mathrm{d}x\right|\\
&\leq \frac{\theta\epsilon}{3\eta(t)\lambda_{1}(t)}\int_{\mathbb{R}}v^{6}\varphi_{l}'\left(\frac{x}{\lambda_{1}(t)}\right)\,\mathrm{d}x+\frac{3\theta}{12 (4\epsilon)^{1/3}\eta(t)\lambda_{1}(t)}\int_{\mathbb{R}}v^{2}\varphi_{l}'\left(\frac{x}{\lambda_{1}(t)}\right)\,\mathrm{d}x\\
&=B_{3,1}(t)+B_{3,2}(t).
\end{split}
\end{equation}
We shall stress that from \eqref{e6.1} it is clear that  independently of $\epsilon,$ the term  $B_{3,2}\in L^{1}\left(\{t\gg 1\}\right).$

Next, we focus our attention on $B_{3,1}.$
In this sense, we proceed  with a modification of the argument   in \cite{km}, so, we consider  a function $\chi:\mathbb{R}\longrightarrow\mathbb{R}$ satisfying $\chi\equiv 1$ on $[0,1]$ and $\chi\equiv0$ on $(-\infty-1]\cup[2,\infty).$

Thus, if we set $\chi_{n}(x):=\chi(x-n),$ then  we get 
 \begin{equation}\label{eq3.1}
\begin{split}
\int_{\mathbb{R}}v^{6}\varphi_{l}'\left(\frac{x}{\lambda_{1}(t)}\right)\,\mathrm{d}x&\leq \sum_{n\in\mathbb{Z}}\left(\int_{\mathbb{R}}v^{6}\chi_{n}^{6}\,\mathrm{d}x\right)\left(\sup_{x\in[n,n+1 ]}\varphi_{l}'\left(\frac{x}{\lambda_{1}(t)}\right)\right).\\
\end{split}
\end{equation}
Next,  by the Gagliardo-Nirenberg-Sobolev inequality  in its optimal form 
\begin{equation}\label{eq1.1}
\begin{split}
\sum_{n\in\mathbb{Z}}\|v\chi_{n}\|_{6}^{6}&\leq C_{\mathrm{opt}}\sum_{n\in\mathbb{Z}}\left(\int_{\mathbb{R}}|\partial_{x}(v\chi_{n})|^{2}\,\mathrm{d}x\right)\|v\chi_{n}\|_{2}^{4}\\
&\leq C_{\mathrm{opt}}\|v\|_{2}^{4}\sum_{n\in\mathbb{Z}}\int_{\mathbb{R}}|\partial_{x}(v\chi_{n})|^{2}\,\mathrm{d}x\\
&\leq 2 C_{\mathrm{opt}}\|v\|_{2}^{4}\sum_{n\in\mathbb{Z}}\int_{\mathbb{R}}(\partial_{x}v)^{2}\chi_{n}^{2}\,\mathrm{d}x + 2 C_{\mathrm{opt}}\|v\|_{2}^{4}\sum_{n\in\mathbb{Z}}\int_{\mathbb{R}}(v\partial_{x}\chi_{n})^{2}\,\mathrm{d}x\\
&\leq 2 C_{\mathrm{opt}}\rho^{2}\sum_{n\in\mathbb{Z}}\int_{\mathbb{R}}(\partial_{x}v)^{2}\chi_{n}^{2}\,\mathrm{d}x +2 cC_{\mathrm{opt}}\rho^{2}\sum_{n\in\mathbb{Z}}\int_{\mathbb{R}}(v\phi_{n})^{2}\,\mathrm{d}x,
\end{split}
\end{equation}
where  
\begin{equation*}
\rho:= \frac{|I_{3}[0]|}{|\alpha|}+\frac{2\gamma}{\alpha}\|u_{0}\|_{2}\|u\|_{L^{\infty}_{t}H^{1}},
\end{equation*}
and $\phi_{n}(x):=\phi(x-n),$  being $\phi$ a $C^{\infty}(\mathbb{R})$  function such that $\phi\equiv 1$ on $[-1,2]$ and $\phi\equiv0 $ on $(-\infty,-2]\cup[3,\infty).$ 

Before start we shall highlight the following relationship given on compacts sets for the weighted function $\varphi_{l}'.$ More precisely,
\begin{equation}\label{eq2.1}
\sup_{x\in[n,n+1]}\varphi_{l}'\left(\frac{x}{\lambda_{1}(t)}\right)\leq \max \left\{e^{-\frac{1}{l\lambda_{1}(t)}},e^{\frac{1}{l\lambda_{1}(t)}}\right\}\inf_{x\in[n,n+1]}\varphi_{l}'\left(\frac{x}{\lambda_{1}(t)}\right).
\end{equation}
Finally, combining \eqref{eq2.1} and \eqref{eq1.1}   we get  from \eqref{eq3.1} that 
\begin{equation*}
\begin{split}
&\int_{\mathbb{R}}v^{6}\varphi_{l}'\left(\frac{x}{\lambda_{1}(t)}\right)\,\mathrm{d}x\\
&\leq    2C_{\mathrm{opt}}\rho^{2} \max \left\{e^{-\frac{1}{l\lambda_{1}(t)}},e^{\frac{1}{l\lambda_{1}(t)}}\right\}\sum_{n\in\mathbb{Z}}\int_{\mathbb{R}}(\partial_{x}v)^{2}\varphi_{l}'\left(\frac{x}{\lambda_{1}(t)}\right)\chi_{n}^{2}\,\mathrm{d}x\\
&\quad +2 cC_{\mathrm{opt}}\rho^{2} \max \left\{e^{-\frac{1}{l\lambda_{1}(t)}},e^{\frac{1}{l\lambda_{1}(t)}}\right\}\sum_{n\in\mathbb{Z}}\int_{\mathbb{R}}v^{2}\varphi_{l}'\left(\frac{x}{\lambda_{1}(t)}\right)\phi_{n}^{2}\,\mathrm{d}x\\
&\leq 2CC_{\mathrm{opt}}\rho^{2} \int_{\mathbb{R}}(\partial_{x}v)^{2}\varphi_{l}'\left(\frac{x}{\lambda_{1}(t)}\right)\,\mathrm{d}x+2C cC_{\mathrm{opt}}\rho^{2} \int_{\mathbb{R}}v^{2}\varphi_{l}'\left(\frac{x}{\lambda_{1}(t)}\right)\,\mathrm{d}x.
\end{split}
\end{equation*}
Therefore,  
\begin{equation*}
|B_{3,1}(t)|\leq\frac{2\theta\epsilon CC_{\mathrm{opt}}\rho^{2}}{3\eta(t)\lambda_{1}(t)} \int_{\mathbb{R}}(\partial_{x}v)^{2}\varphi_{l}'\left(\frac{x}{\lambda_{1}(t)}\right)\,\mathrm{d}x+\frac{2\theta\epsilon C cC_{\mathrm{opt}}\rho^{2}}{3\eta(t)\lambda_{1}(t)} \int_{\mathbb{R}}v^{2}\varphi_{l}'\left(\frac{x}{\lambda_{1}(t)}\right)\,\mathrm{d}x.
\end{equation*}
Note  that the second  term on the r.h.s is bounded after integrating in time, this is a consequence of  \eqref{e6.1}. To handle the first term we notice  
 that this is  up to constant the quantity to be estimated after integrating in time. Thus,  for our proposes it is enough to take 
 $\epsilon>\frac{3}{CC_{\mathrm{opt}}\rho^{2}}$. 

Next, we   decompose as follows: 
For all  $\epsilon_{1}>0,$
 \begin{equation}\label{e1.3}
 \begin{split}
 |B_{11}(t)|&=\frac{|\mu\beta|}{2{\eta(t)}\lambda_{1}(t)}\int_{\mathbb{R}}|u|^{4}\varphi_{l}'\left(\frac{x}{\lambda_{1}(t)}\right)\,\mathrm{d}x\\
 &\leq\frac{\epsilon_{1}\mu\beta}{2{\eta(t)}\lambda_{1}(t)} \int_{\mathbb{R}}|u|^{2}\varphi_{l}'\left(\frac{x}{\lambda_{1}(t)}\right)\,\mathrm{d}x+\frac{\mu\beta}{8\epsilon_{1}{\eta(t)}\lambda_{1}(t)}\int_{\mathbb{R}}|u|^{6}\varphi_{l}'\left(\frac{x}{\lambda_{1}(t)}\right)\,\mathrm{d}x\\
 &=B_{11,1}(t)+B_{11,2}(t).
 \end{split}
 \end{equation}
 \begin{remark}
 	In the case $\beta>0,$ the decomposition \eqref{e1.3} is not required, since  the integral of $B_{11}$ is obtained   by directly from the energy estimate \eqref{e2}.  
 \end{remark}
 Note that independent of the  value  of $\epsilon_{1}$, the bound in   \eqref{e6.1} yield
 \begin{equation*}
 \int_{\{t\gg1\}}|B_{11,1}(t)|\mathrm{d}t<\infty.
 \end{equation*}

 Thus, by using the same notation as in \eqref{eq3.1} it follows that
 \begin{equation}\label{eq3}
 \begin{split}
 \int_{\mathbb{R}}|u|^{6}\varphi_{l}'\left(\frac{x}{\lambda_{1}(t)}\right)\,\mathrm{d}x&\leq \sum_{n\in\mathbb{Z}}\left(\int_{\mathbb{R}}|u|^{6}\chi_{n}^{6}\,\mathrm{d}x\right)\left(\sup_{x\in[n,n+1 ]}\varphi_{l}'\left(\frac{x}{\lambda_{1}(t)}\right)\right)\\
 \end{split}
 \end{equation}
 whence we get by the Gagliardo-Nirenberg-Sobolev inequality that 
 \begin{equation}\label{eq1}
 \begin{split}
 &\sum_{n\in\mathbb{Z}}\|u\chi_{n}\|_{6}^{6}\\
 &\leq C_{\mathrm{opt}}\sum_{n\in\mathbb{Z}}\left(\int_{\mathbb{R}}|\partial_{x}(u\chi_{n})|^{2}\mathrm{d}x\right)\|u\chi_{n}\|_{2}^{4}\\
 &\leq C_{\mathrm{opt}}\|u_{0}\|_{2}^{4}\sum_{n\in\mathbb{Z}}\int_{\mathbb{R}}|\partial_{x}(u\chi_{n})|^{2}\mathrm{d}x\\
 &\leq 2C_{\mathrm{opt}}\|u_{0}\|_{2}^{4}\sum_{n\in\mathbb{Z}}\int_{\mathbb{R}}|\partial_{x}u|^{2}\chi_{n}^{2}\mathrm{d}x+ 2 C_{\mathrm{opt}}\|u_{0}\|_{2}^{4}\sum_{n\in\mathbb{Z}}\int_{\mathbb{R}}|u\partial_{x}\chi_{n}|^{2}\mathrm{d}x\\
 &\leq 2 C_{\mathrm{opt}}\|u_{0}\|_{2}^{4}\sum_{n\in\mathbb{Z}}\int_{\mathbb{R}}|\partial_{x}u|^{2}\chi_{n}^{2}\mathrm{d}x +2 cC_{\mathrm{opt}}\|u_{0}\|_{2}^{4}\sum_{n\in\mathbb{Z}}\int_{\mathbb{R}}|u\phi_{n}|^{2}\mathrm{d}x,
 \end{split}
 \end{equation}
 where $\phi_{n}(x):=\phi(x-n)$  being $\phi$ a $C^{\infty}(\mathbb{R})$  function such that $\phi\equiv 1$ on $[-1,2]$ and $\phi\equiv0 $ on $(-\infty,-2]\cup[3,\infty).$ 
 
 We shall remark that 
\begin{equation}\label{eq2}
\sup_{x\in[n,n+1]}\varphi_{l}'\left(\frac{x}{\lambda_{1}(t)}\right)\leq \max \left\{e^{-\frac{1}{l\lambda_{1}(t)}},e^{\frac{1}{l\lambda_{1}(t)}}\right\}\inf_{x\in[n,n+1]}\varphi_{l}'\left(\frac{x}{\lambda_{1}(t)}\right).
\end{equation}
Finally, combining \eqref{eq2} and \eqref{eq1}   we get  from \eqref{eq3} that 
 \begin{equation*}
\begin{split}
&\int_{\mathbb{R}}|u|^{6}\varphi_{l}'\left(\frac{x}{\lambda_{1}(t)}\right)\,\mathrm{d}x\\
&\leq    2C_{\mathrm{opt}}\|u_{0}\|_{2}^{4} \max \left\{e^{-\frac{1}{l\lambda_{1}(t)}},e^{\frac{1}{l\lambda_{1}(t)}}\right\}\sum_{n\in\mathbb{Z}}\int_{\mathbb{R}}|\partial_{x}u|^{2}\varphi_{l}'\left(\frac{x}{\lambda_{1}(t)}\right)\chi_{n}^{2}\,\mathrm{d}x\\
&\quad 2 cC_{\mathrm{opt}}\|u_{0}\|_{2}^{4} \max \left\{e^{-\frac{1}{l\lambda_{1}(t)}},e^{\frac{1}{l\lambda_{1}(t)}}\right\}\sum_{n\in\mathbb{Z}}\int_{\mathbb{R}}|u|^{2}\varphi_{l}'\left(\frac{x}{\lambda_{1}(t)}\right)\phi_{n}^{2}\,\mathrm{d}x\\
&\leq 2CC_{\mathrm{opt}}\|u_{0}\|_{2}^{4} \int_{\mathbb{R}}|\partial_{x}u|^{2}\varphi_{l}'\left(\frac{x}{\lambda_{1}(t)}\right)\,\mathrm{d}x\\
&\quad +2C cC_{\mathrm{opt}}\|u_{0}\|_{2}^{4} \int_{\mathbb{R}}|u|^{2}\varphi_{l}'\left(\frac{x}{\lambda_{1}(t)}\right)\,\mathrm{d}x.
\end{split}
\end{equation*}
The same analysis can be applied to estimate the term $B_{2,2}$. 

Therefore, after  choosing  $\epsilon_{1}>0,$ satisfying $
\epsilon_{1}>2\beta CC_{\mathrm{opt}}\|u_{0}\|_{2}^{4},$
we get after  gathering the estimates corresponding to this step that  
\begin{equation}\label{e6}
\int\limits_{\{t\gg1\}}\Big(\frac{1}{\eta(t)\lambda_{1}(t)}\int\limits_{\mathbb{R}}\left(|\partial_{x}u(t,x)|^{2}+(\partial_{x}v(t,x))^{2}\Big)\varphi_{l}'\left(\frac{x}{\lambda_{1}(t)}\right)\,\mathrm{d}x\right)\mathrm{d}t<\infty;
\end{equation}

that  implies  
\begin{equation*}
\liminf_{t\rightarrow \infty}\int_{\Omega_{p_{1}}(t)}\left|\partial_{x}u(t,x)\right|^{2}\,\mathrm{d}x=\liminf_{t\rightarrow \infty}\int_{\Omega_{p_{1}}(t)} \left(\partial_{x}v(t,x)\right)^{2}\,\mathrm{d}x=0,
\end{equation*}
whenever $0<p_{1}\leq \frac{2}{p_{2}+2}$ for  $p_{2}>1.$
\end{proof}

\vspace{0.5cm}

\noindent{\bf Acknowledgements.}  The first author was partially supported by CNPq grant 305791/2018-4 and FAPERJ grant E-26/202.638/2019. The second author was partially supported by CMM Conicyt Proyecto Basal AFB170001. 
The authors are grateful to John Albert and Filipe Oliveira for useful information.

\end{document}